\documentclass[12pt,oneside]{article}
\usepackage{amsmath,amssymb,amsfonts,amsthm}
\usepackage{color}
\usepackage[all]{xy}
\usepackage{graphicx}
\usepackage{color}
\usepackage{makeidx}
\usepackage{multirow}
\usepackage{rotating}
\newcommand{\T}{{\cal T}}

\newcommand{\Real}{\mathbb R}

\newcommand{\To}{\longrightarrow}

\newcommand {\cppp}{\mathfrak{X}(TM)}
\newcommand {\N}{\mathcal{N}}

\def\Section#1{\vspace{30truept}\addtocounter{section}{1}\setcounter{thm}{0}\setcounter{equation}{0}
{\noindent\Large\bf\arabic{section}.~~#1}\par \vspace{12pt}}

\newtheorem{thm}{Theorem}[section]
\newtheorem{cor}[thm]{Corollary}
\newtheorem{lem}[thm]{Lemma}
\newtheorem{prop}[thm]{Proposition}
\newtheorem{defn}[thm]{Definition}
\newtheorem{example}[thm]{Example}
\newtheorem{rem}[thm]{Remark}

\numberwithin{equation}{section}

\newcommand\undersym[2]{\raisebox{-7pt}{\tiny$#2$}{\kern-8pt}\mbox{$#1$}}
\newcommand\undersymm[2]{\raisebox{-8pt}{\tiny$#2$}{\kern-15pt}\mbox{$#1$}}
\newcommand\overast[1]{\raisebox{10pt}{\small$\ast$}{\kern-7.5pt}\mbox{$#1$}}
\newcommand\overlind[1]{\raisebox{10pt}{\small$\overline{{\hspace{2pt}}\star}$}{\kern-7.5pt}\mbox{$#1$}}
\newcommand\overlinc[1]{\raisebox{10pt}{\tiny$\overline{{\hspace{2pt}}\circ}$}{\kern-7.5pt}\mbox{$#1$}}
\newcommand\overlina[1]{\raisebox{10pt}{\small$\overline{{\hspace{1pt}}\ast}$}{\kern-7.5pt}\mbox{$#1$}}
\newcommand\overcirc[1]{\raisebox{10pt}{\tiny{$\circ$}}{\kern-7.5pt}\mbox{$#1$}}
\newcommand\overdiamond[1]{\raisebox{10pt}{\small$\star$}{\kern-7.5pt}\mbox{$#1$}}

\begin{document}
\title{\bf{Nullity distributions associated to \\Cartan connection}\footnote{ArXiv Number: }}
\author{\bf{ Nabil L. Youssef$^{\,1,2}$, A.
Soleiman$^{3}$ and S. G.
Elgendi$^{3}$}}
\date{}
\maketitle                     
\vspace{-1.16cm}
\begin{center}
{$^{1}$Department of Mathematics, Faculty of Science,\\ Cairo
University, Giza, Egypt}\\

\bigskip
$^{2}$Center for Theoretical Physics (CTP)\\
at the Britich University in Egypt (BUE)
\end{center}

\begin{center}
{$^{3}$Department of Mathematics, Faculty of Science,\\ Benha
University, Benha,
 Egypt}
\end{center}

\begin{center}
E-mails: nlyoussef@sci.cu.edu.eg, nlyoussef2003@yahoo.fr\\
{\hspace{2.4cm}}amr.hassan@fsci.bu.edu.eg, amrsoleiman@yahoo.com\\
{\hspace{1.8cm}}salah.ali@fsci.bu.edu.eg, salahelgendi@yahoo.com
\end{center}

\bigskip
\vspace{1cm} \maketitle
\smallskip
\noindent{\bf \begin{center}
Abstract
\end{center}}

The Klein-Grifone approach to global Finsler geometry is adopted. The nullity distributions of the three curvature tensors of Cartan connection are investigated. Nullity distributions concerning certain relevant special Finsler spaces are considered. Concrete examples are given whenever the situation needs.

\bigskip
\vspace{7pt}

\newpage

\Section{Introduction}

 Chern and Kuiper \cite{chern.k}  in 1952  defined a distribution on a Riemannian manifold M which assigns  to  each point $x\in M$,  the  subspace
$$\N_{R}(x) =  \{X\in  T_xM : R(X, Y)  = 0 , \, \forall \, Y\in T_xM\},$$
where ${R}$  is the curvature of  the Riemannian connection on $M$.   It  is called the
nullity space at $x$.  The distribution defined by the  subspace  $\N_x$ at each point $x$
of $M$ is called the nullity distribution $\N$ of  the Riemannian manifold $M$. The dimension $\mu_x$ of $\N_x$ is called the  index of nullity at $x$.  Chern
and Kuiper showed that  if $\mu_x$  is constant in a neighborhood then $\N$  constitutes
a completely  integrable distribution there, and that  the leaves  of  the  resulting
foliation are flat.  Later, Maltz  \cite{Maltz}  showed that  the  leaves are also totally  geodesic.

In 1972, Akbar Zadeh \cite{akbar.null.}, \cite{akbar.null.2} extended this work to  Finsler geometry adopting the \textbf{pullback approach} to global Finsler geometry. He studied the nullity distribution of the (classical) curvature of Cartan connection. Recently, Bidabad and  Refie-Rad \cite{bidabad} studied a more general case called k-nullity distribution in Finsler geometry.
On the other hand, in 1982, Youssef \cite{Nabil.2},  \cite{Nabil.1} studied the nullity distributions of the curvature tensors of Barthel
connection and Berwald connection, adopting the \textbf{Klein-Grifone approach }to global Finsler geometry.

In the present paper, we investigate the nullity distribution of the three curvature tensors of Cartan connection adopting the \textbf{Klein-Grifone approach} \cite{r21},\,\cite{r22}, and \cite{r27}. The paper is  organized as follows. In the first section, we give the necessary material that will be needed throughout
the present work. In particular, we give a brief account  on the Klein-Grifone approach to global Finsler geometry.
In the second section, we focus our attention on the most important properties and formulas related to the curvature tensors of  Cartan connection.
In the third section, we investigate  the nullity distribution (ND) $\N_R$ of the h-curvature tensor $R$ of Cartan connection, the nullity spaces being subspaces of the horizontal space. We show that the ND $\N_R$ is included in the ND of the curvature of Barthel connection and we show, by an example, that this inclusion is proper. We show that the ND $\N_R$ is  completely integrable and  the leaves of the nullity foliation are  auto-parallel and hence totally geodesic submanifolds.
In the Fourth and fifth sections, we study  the ND's  of the hv-curvature and  v-curvature of Cartan connection.   We show through examples that these ND's are  not completely integrable. Nevertheless, we   investigate necessary and sufficient conditions for such distributions  to be completely integrable.

It should be noted that in the pullback approach (\cite{akbar.null.}, \cite{akbar.null.2}) the ND of the classical curvature of Cartan connection is  completely integrable and, consequently, the ND's of the h-curvature, hv-curvature and v-curvature are completely integrable. However, in the Klein-Grifone approach the situation is different: the ND of the h-curvature is completely integrable whereas the ND's of the hv-curvature and v-curvature are not.

Throughout the paper, we give concrete examples whenever the situation needs. Moreover, we study ND's related to certain special Finsler spaces relevant to the situation under consideration.

\Section{Notation and Preliminaries}
In this section, we give a brief account of the basic concepts
 of the Klein-Grifone approach to global Finsler geometry. For more
 details, we refer to \cite{r21},\,\cite{r22}, and \cite{r27}.
 We make the
assumption that the geometric objects we consider are of class
$C^{\infty}$.\\ The
following notations will be used throughout this paper:\\
 $M$: a real differentiable manifold of finite dimension $n$ and of
class $C^{\infty}$,\\
  $\mathfrak{F}(M)$: the $\Real$-algebra of differentiable functions
on $M$,\\
 $\mathfrak{X}(M)$: the $\mathfrak{F}(M)$-module of vector fields
on $M$,\\
$\pi_{M}:TM\longrightarrow M$: the tangent bundle of $M$,\\
$\pi: \T M\longrightarrow M$: the subbundle of nonzero vectors
tangent to $M$,\\
$V(TM)$: the vertical subbundle of the bundle $TTM$,\\
$ i_{X}$ : the interior product with respect to  $X
\in\mathfrak{X}(M)$,\\
$df$ : the exterior derivative  of $f$,\\
$ d_{L}:=[i_{L},d]$, $i_{L}$ being the interior derivative with
respect to a vector form $L$,\\
$\mathcal{L}_X$ : the Lie derivative  with respect to $X\in\mathfrak{{X}}(M)$.\\

We have the following short exact sequence of vector bundles,
relating the tangent bundle $T(T M)$ and the pullback bundle
$\pi^{-1}(TM)$:\vspace{-0.1cm}
$$0\longrightarrow
 \pi^{-1}(TM)\stackrel{\gamma}\longrightarrow T(TM)\stackrel{\rho}\longrightarrow
\pi^{-1}(TM)\longrightarrow 0 ,\vspace{-0.1cm}$$
 where the bundle morphisms $\rho$ and $\gamma$ are defined respectively by
$\rho := (\pi_{\T M},d\pi)$ and $\gamma (u,v):=j_{u}(v)$, where
$j_{u}$  is the natural isomorphism $j_{u}:T_{\pi_{M}(v)}M
\longrightarrow T_{u}(T_{\pi_{M}(v)}M)$. The vector $1$-form $J$ on
$TM$ defined by $J:=\gamma\circ\rho$ is called the natural almost
tangent structure of $T M$. The vertical vector field ${C}$
on $TM$ defined by ${C}:=\gamma\circ\overline{\eta}$, where $\overline{\eta}$
is the vector field on $\pi^{-1}(TM)$ given by $\overline{\eta}(u)=(u,u)$, is
called the fundamental or the canonical (Liouville) vector field.

\par

\par

In this work, we shall need the evaluation of the
Fr\"{o}licher-Nijenhuis  bracket  in some special cases \cite{r20}:

 If $L$ is a vector $\ell$-form  and $X\in \mathfrak{X}(M)$, then, for all $Y_{1}, ..., Y_{\ell}\in
\mathfrak{X}(M)$,
$$[X,L](Y_{1},...,Y_{\ell})=[X, L(Y_{1},...,Y_{\ell})]-\sum_{i=1}^{\ell}
L(Y_{1},..., [X,Y_{i}],...,Y_{\ell}).$$
In particular, if  $L$ is  vector $1$-form,
$$[X,L]Y=[X, LY]-L[X,Y].$$

 \ If $K$ and $L$ are vector $1$- forms, then
\begin{eqnarray*}
  [K,L](X,Y)&=& [KX,LY]+[LX,KY]+K L[X,Y]+L K[X,Y] \\
   && -K[L X,Y]-K[ X,L Y]-L[K X,Y]-L[ X,K Y].
\end{eqnarray*}
In particular, the vector $2$-form $N_{K}:=\frac{1}{2}[K,K]$ is
said to be the Nijenhuis torsion of the vector $1$-form $K$:
\begin{equation}\label{Nk}
     N_{K}:= \frac{1}{2}[K,K](X,Y)=[KX,KY]+K^{2}[X,Y] -K[K X,Y]-K[ X,K
     Y].
\end{equation}

One can show that the natural almost tangent structure $J$ has the
properties:
\begin{equation}\label{JJ}
  J^{2}=0, \ \ [J,J]=0, \ \  [{C}, J] = -J,
  \ \ \text{Im}(J) = Ker (J) = V (T M),
\end{equation}

 A scalar  $p$-form $\omega$ on $TM$ is semi-basic if $ i_{JX}\omega=0,\,\, \forall X\in \cppp$. A vector $\ell$-form $L$ on $TM$ is semi-basic if $ JL=0\,\, \text{and}\,\, i_{JX}L=0, \,\, \forall X\in \cppp$.

 A scalar  $p$-form $\omega$ on $TM$ is  homogenous of degree r  if $\mathcal{L}_C\omega=r\omega$. A vector $\ell$-form $L$ on $TM$ is homogenous of degree r, denoted h(r), if $[C,L]=(r-1)L$. It is clear that $J$ is h(0).

 A semispray  on $M$ is a vector field $S$ on $TM$,
 $C^{\infty}$ on $\T M$, $C^{1}$ on $TM$, such that
$JS = C$. A semispray $S$ which is
homogeneous of degree $2$
($[{C},S]= S $) is called a spray.

A nonlinear connection on $M$ is a vector $1$-form $\Gamma$ on $TM$,
$C^{\infty}$ on $\T M$, $C^{0}$ on $TM$, such that
$$J \Gamma=J, \quad\quad \Gamma J=-J .$$
The vertical and horizontal   projectors $v$\,  and
$h$ associated with $\Gamma$ are defined respectively  by
   $v:=\frac{1}{2}
 (I-\Gamma),\, h:=\frac{1}{2} (I+\Gamma).$
Thus $\Gamma$ gives rise to the direct sum decomposition $TT M=
V(TM)\oplus H(TM)$, where $\,\,V(TM):= Im \, v=Ker \, h$ is the vertical bundle and   $H(TM):=Im \, h = \ker\,
v $ is the horizontal bundle induced by $\Gamma$. An element of $V(TM)$ (resp. $H(TM)$) will be denoted by $vX$  (resp. $hX$). We
have $J v=0, \,\,\, v J=J, \,\,\, J h =J, \,\,\, h J=0.$  A nonlinear
connection $\Gamma$ is homogeneous if $[{C},\Gamma]=0$.  To each  nonlinear connection $\Gamma$,
one can associate a semispray $S$ which is horizontal with respect
to $\Gamma$, namely, $S=hS'$, where $S'$ is an arbitrary
semispray. Moreover, if $\Gamma$ is homogeneous, then its associated
semispray is a spray.

The torsion $t$ of a nonlinear connection $\Gamma$ is the vector
$2$-form  on $TM$ defined by $t:=\frac{1}{2} [J,\Gamma]$.
The curvature of $\Gamma$ is the vector $2$-form on $TM$ defined by
$\mathfrak{R}:=-\frac{1}{2}[h,h]$. Associated with $\Gamma$,
an almost complex structure $F$ $(F^2=-I)$ is  defined by $FJ=h$ and $Fh=-J$. This
$F$ defines an isomorphism of $T_zTM$  for all $z\in TM$.   \\

\begin{defn}\label{Finsler} \emph{\cite{r27}} A Finsler space
 of dimension $n$ is a pair $(M,E)$,
where $M$ is a
 differentiable manifold of dimension $n$ and $E$ is a map
 $$E: TM \To \Real ,\vspace{-0.1cm}$$
called the energy function,  satisfying the axioms{\em:}
 \begin{description}
    \item[\em{\textbf{(a)}}] $E(u)>0 $ for all $u\in \T M$ and $E(0)=0$,
    \item[\em{\textbf{(b)}}] $E $ is  $C^{\infty}$ on  $\T M$, $C^{1}$ on $TM$,
    \item[\em{\textbf{(c)}}]$E$ is homogenous of degree $2${\em:}
    $\mathcal{L}_{{C}} E=2E$,
    \item[\em{\textbf{(d)}}] The exterior  $2$-form
    $\Omega:=dd_{J}E$, called the fundamental form,  has maximal rank.
     \end{description}
 \end{defn}
\begin{thm}{\em{\cite{r27}}}\label{spray} Let $(M,E)$ be a Finsler space. The vector field $S\in \cppp$
 defined by $i_{S}\Omega =-dE$ is a spray.
 Such a spray is called the canonical spray  associated with $(M,E)$.
 \end{thm}
\par
   Now, we give a fundamental result which ensures the
existence and uniqueness of a remarkable nonlinear connection.
\begin{thm}{\em{\cite{r27}}}\label{Barthel} On a Finsler space $(M,E)$, there exists a unique
conservative \emph{(}$d_hE=0$\emph{)} homogeneous nonlinear  connection with zero torsion.
It is given by{\em:} \vspace{-0.1cm}
$$\Gamma = [J,S],\vspace{-0.1cm}$$
where $S$ is the canonical spray.
Such a connection is called the canonical connection,  Barthel
connection  or  Cartan
nonlinear connection
associated with $(M,E)$.
\end{thm}

It should be noted that the semi-spray associated with the Barthel
connection is a spray, which is  the canonical spray.

\Section{Berwald  and Cartan connections}

In this section, we present the necessary material, concerning Berwald and Cartan connections, that will be needed throughout
the present work.    For more details, we refer to \cite{r22} and \cite{Nabil.1}.

\begin{thm} \cite{r22} For a Finsler space  $(M,E)$,
there exists a unique linear connection  \, $\overcirc{D}$ on $TM$  satisfying the following properties:
\begin{description}
                  \item[(a)]$\overcirc{D}J=0$.\hspace{5.4cm}\em{\textbf{(b)}}\, $\overcirc{D}C=v$.
                  \item[(c)] $\overcirc{D}\Gamma=0\,\,
                  (\Longleftrightarrow\overcirc{D}h=\overcirc{D}v=0  )$. \hspace{1.7cm}{\textbf{(d)}}\, $\overcirc{D}_{JX}JY=J[JX,Y]$.
                  \item[(e)]$\overcirc{T}(JX,Y)=0$,
 \end{description}
 where $h$ and $v$ are the horizontal and vertical projectors
of Barthel connection and \, $\overcirc{T}$ is the (classicl) torsion of \, $\overcirc{D}$. This connection is called the Berwald connection.
\end{thm}
\par
The explicit expression  of\,  $\overcirc{D}$ is given by:
\begin{equation}\label{berwaldconn.}
  \left.
    \begin{array}{rcl}
  \overcirc{D}_{JX}JY&=&J[JX,Y],\\
\overcirc{D}_{hX}JY&=&v[hX,JY],\\
  \overcirc{D}F&=&0.
 \end{array}
  \right\}
\end{equation}
\begin{lem}
The  Berwald connection has the property that
$$\overcirc{T}(hX,hY)=\mathfrak{R}(X,Y),$$
where $\mathfrak{R}$ is the curvature of Barthel connection.
\end{lem}

\par
Let $(M,E)$ be a  Finsler space and  $\Omega:=dd_{J}E$. The map $\overline{g}$ defined by
$$\overline{g}(J X,J Y):=\Omega(JX,Y), \ \forall \ X, Y \in  T(TM)$$
defines a  metric  on $V(TM)$. This metric   can be extended to a metric $g$ on $T(TM)$ defined by the formula:

 \begin{equation}\label{metricg}
 g(X,Y)=\overline{g}(JX,JY)+\overline{g}(vX,vY)=\Omega(X,FY).
\end{equation}

\begin{thm}\cite{r22} For a Finsler space $(M,E)$, there exists a unique linear connection  ${D}$ on $TM$ satisfying the following properties:
\begin{description}
                  \item[(a)]${D}J=0$.\hspace{4.4cm} \em{\textbf{(b)}} ${D}C=v$.
                  \item[(c)] ${D}\Gamma=0 \,\,(\Longleftrightarrow {D}h={D}v=0 )$.\hspace{.8cm}\textbf{(d)} ${D}g=0$.
                  \item[(e)] ${T}(JX,JY)=0$.\hspace{3.25cm}\textbf{(f)} $JT(hX,hY)=0$.
 \end{description}
 This connection  is called the Cartan  connection.
 \end{thm}
 \par
The explicit expression  of\,  ${D}$ is given by:
\begin{equation}\label{cartanconn.}
  \left.
    \begin{array}{rcl}
  D_{JX}JY&=&\overcirc{D}_{JX}JY+\mathcal{C}(X,Y),\\
  D_{hX}JY&=&\overcirc{D}_{hX}JY+\mathcal{C}'(X,Y),\\
  {D}F&=&0,
\end{array}
  \right\}
\end{equation}
where $\mathcal{C}$ and $\mathcal{C}'$ are the scalar 2-forms on $TM$ defined  by
$$\Omega(\mathcal{C}(X,Y),Z)=\frac{1}{2}(\mathcal{L}_{JX}(J^\ast g))(Y,Z),\quad\quad
\Omega(\mathcal{C}'(X,Y),Z)=\frac{1}{2}(\mathcal{L}_{hX}g)(JY,JZ), $$
with  $(J^\ast g)(Y,Z)=g(JY,JZ)$.

 The tensors  ${\mathcal{C}}$ and $\mathcal{C}'$ will be called the first and second Cartan tensors respectively. They  are semi-basics,
 symmetric and
  \begin{equation}\label{c(s)}
  {\mathcal{C}}(X,S)=\mathcal{C}'(X,S)=0.
  \end{equation}

\bigskip

We have the following lemmas.
  \begin{lem}
The (h)h-torsion ${T}(hX,hY)$ and (h)v-torsion $T(hX,JY)$ of Cartan connection are  given respectively by
$${T}(hX,hY)=\mathfrak{R}(X,Y),\quad T(hX,JY)=(\mathcal{C}'-F\mathcal{C})(X,Y),$$
where $\mathfrak{R}$ is the curvature of Barthel connection.
\end{lem}
\begin{lem}\label{car.r,p,q}
The h-curvature $R$, hv-curvature $P$  and v-curvature $Q $ of Cartan connection are given respectively by:
\begin{description}
  \item[(a)]$ R(X,Y)Z  = \overcirc{R}(X,Y)Z+(D_{hX}\mathcal{C}')(Y,Z)-(D_{hY}\mathcal{C}')(X,Z)
      +\mathcal{C}'(F\mathcal{C}'(X,Z),Y)\\
      {\hspace{1.9cm}}-\mathcal{C}'(F\mathcal{C}'(Y,Z),X)+\mathcal{C}(F\mathfrak{R}(X,Y),Z). $
  \item[(b)] $ P(X,Y)Z  = \overcirc{P}(X,Y)Z+(D_{hX}{\mathcal{C}})(Y,Z)-(D_{JY}\mathcal{C}')(X,Z)
      +\mathcal{C}(F\mathcal{C}'(X,Z),Y) \\
      {\hspace{1.9cm}} +{\mathcal{C}}(F\mathcal{C}'(X,Y),Z)-\mathcal{C}'(F{\mathcal{C}}(Y,Z),X)-\mathcal{C}'(F{\mathcal{C}}(X,Y),Z). $
  \item[(c)] $Q(X,Y)Z={\mathcal{C}}(F{\mathcal{C}}(X,Z),Y)-{\mathcal{C}}(F{\mathcal{C}}(Y,Z),X),$
\end{description}
where\, $\overcirc{R}$ and\, $\overcirc{P}$ are respectively the h-curvature and hv-curvature
  of Berwald connection.
\end{lem}
\begin{lem}\label{car.curv.}
For Cartan connection,  the following properties hold:
\begin{description}
  \item[(a)]$R(X,Y)S=\mathfrak{R}(X,Y)$.
  \item[(b)] $P(X,Y)S=\mathcal{C}'(X,Y)$.
  \item[(c)]$P(S,X)Y=P(X,S)Y=0$.
  \item[(d)]$Q(S,X)Y=Q(X,S)Y=Q(X,Y)S=0.$
\end{description}
\end{lem}
\begin{lem}\label{bianchi}
The Bainchi identities for Cartan connection are given by:
\begin{description}
  \item[(a)]$\mathfrak{S}_{X,Y,Z}\{R(X,Y)Z\}=\mathfrak{S}_{X,Y,Z}\{\mathcal{C}(F\mathfrak{R}(X,Y),Z)\}$.

  \item[(b)] $\mathfrak{S}_{X,Y,Z}\{Q(X,Y)Z\}=0$.

  \item[(c)] $\mathcal{C}(F\mathfrak{R}(X,Y),Z)=\mathfrak{R}(F\mathcal{C}(X,Z),Y)-\mathfrak{R}(F\mathcal{C}(Y,Z),X)$.

  \item[(d)] $\mathfrak{S}_{X,Y,Z}\{(D_{hX}\mathfrak{R})(Y,Z)\}=\mathfrak{S}_{X,Y,Z}\{
  \mathcal{C}'(F\mathfrak{R}(X,Y),Z)\}$.

  \item[(e)]$\mathfrak{S}_{X,Y,Z}\,\{(D_{hX}R)(Y,Z)\}=\mathfrak{S}_{X,Y,Z}\,\{P(X,F\mathfrak{R}(Y,Z))\}$.

  \item[(f)]$(D_{hX}P)(Y,Z)-(D_{hY}P)(X,Z)+(D_{JZ}R)(X,Y)=P(X,F\mathcal{C}'(Y,Z))\\
  -P(Y,F\mathcal{C}'(X,Z))+R(F{\mathcal{C}}(Y,Z),X)-R(F{\mathcal{C}}(X,Z),Y)-Q(F\mathfrak{R}(X,Y),Z)$.

  \item[(g)]$(D_{hX}Q)(Y,Z)-(D_{JY}P)(X,Z)+(D_{JZ}P)(X,Y)=P(F{\mathcal{C}}(X,Y),Z)\\
  -P(F{\mathcal{C}}(Z,X),Y)-Q(F\mathcal{C}'(X,Y),Z)
  +Q(F\mathcal{C}'(Z,X),Y)$.

  \item[(h)]$\mathfrak{S}_{X,Y,Z}\{(D_{JX}Q)(Y,Z)\}=0$,
\end{description}
where $\mathfrak{S}_{X,Y,Z}$ is the cyclic sum over the vector fields $X$, $Y$ and $Z$.
\end{lem}
\Section{Nullity distribution of Cartan h-curvature}

We are now in a position to study the nullity distributions associated to Cartan connection. Firstly, we  study the nullity
 distribution of the h-curvature tensor. It should be noted that the nullity distributions of  Barthel and Berwald connections have been investigated in \cite{Nabil.2} and \cite{Nabil.1}.\\

We need the following  lemma for subsequent use.

\begin{lem}\label{car.[]}For all $X,Y\in \cppp$, we have
\begin{description}
  \item[(a)] $[JX,JY]=J(D_{JX}Y-D_{JY}X).$
  \item[(b)] $[hX,JY]=J(D_{hX}Y)-h(D_{JY}X)-(\mathcal{C}'-F\mathcal{C})(X,Y).$
  \item[(c)] $[hX,hY]=h(D_{hX}Y-D_{hY}X)-\mathfrak{R}(X,Y).$
\end{description}
\end{lem}

\begin{proof}~\par
\noindent \textbf{(a)} Using (\ref{berwaldconn.}) and (\ref{cartanconn.}), by the symmetry of $\mathcal{C}$ and since $[J,J]=0$, $J^2=0$ and $DJ=0$, we get
\begin{eqnarray*}
   J(D_{JX}Y-D_{JY}X)&=&D_{JX}JY-D_{JY}JX  \\
   &=&\overcirc{D}_{JX}JY+\mathcal{C}(X,Y)-\overcirc{D}_{JY}JX-\mathcal{C}(Y,X) \\
   &=&J[JX,Y]-J[JY,X]\\
   &=&[JX,JY].
 \end{eqnarray*}
\noindent \textbf{(b)} Using (\ref{berwaldconn.}) and (\ref{cartanconn.}), by  the symmetry  of $\mathcal{C}$ and since $DJ=Dh=DF=0$, we obtain
\begin{eqnarray*}
   J(D_{hX}Y)-h(D_{JY}X)&=&D_{hX}JY-D_{JY}hX  \\
   &=&\overcirc{D}_{hX}JY+\mathcal{C}'(X,Y)-\overcirc{D}_{JY}hX-F\mathcal{C}(Y,X) \\
   &=&v[hX,JY]-h[JY,X]+(\mathcal{C}'-F\mathcal{C})(X,Y)\\
     &=&[hX,JY]+(\mathcal{C}'-F\mathcal{C})(X,Y).
 \end{eqnarray*}
 \noindent \textbf{(c)} Again using (\ref{berwaldconn.}) and (\ref{cartanconn.}), by   the symmetry property of $\mathcal{C}'$, we have
\begin{eqnarray*}
   h(D_{hX}Y-D_{hY}X)&=&D_{hX}hY-D_{hY}hX  \\
   &=&\overcirc{D}_{hX}hY+F\mathcal{C}'(X,Y)-\overcirc{D}_{hY}hX-F\mathcal{C}'(Y,X) \\
      &=&Fv[hX,JY]+Fv[JX,hY].
    \end{eqnarray*}
 As the torsion of $\Gamma$  vanishes, then $0=t(X,Y)=v[JX,hY]+v[hX,JY]-J[hX,hY]$, from which  $Fv[JX,hY]+Fv[hX,JY]=FJ[hX,hY]=h[hX,hY]$. Consequently,
 $$ h(D_{hX}Y-D_{hY}X) =h[hX,hY]=[hX,hY]-v[hX,hY]=
   [hX,hY]+\mathfrak{R}(X,Y),$$
  where we have used the identity $\mathfrak{R}(X,Y)=-v[hX,hY]$ \cite{Nabil.2}.
\end{proof}
\begin{rem}\label{barthelr}\em{It is to be noted that the identity $ \mathfrak{R}(X,Y)=-v[hX,hY]$ shows that the Lie bracket of two horizontal  vector fields is horizontal if and only if the curvature $\mathfrak{R}$ vanishes. This means that a necessary and sufficient condition for the horizontal distribution to be completely integrable is that  $\mathfrak{R}$ vanishes. This fact can also be deduced from Lemma \ref{car.[]} \textbf{(c)} above.}
\end{rem}

\begin{defn}\label{nr} Let $R$ be the h-curvature tensor of Cartan connection.
The nullity space of $R$ at a point $z\in TM$ is the subspace of $H_z(TM)$ defined by
$$\mathcal{N}_R(z):=\{X\in H_z(TM) : \,  R(X,Y)=0, \, \,\forall\, Y\in T_z(TM)\}.$$
The dimension of $\mathcal{N}_R(z)$, denoted by $\mu_R(z)$, is the index of nullity of $R$ at $z$.

 If the index of nullity is constant,
then the map $\mathcal{N}_R:z\mapsto \mathcal{N}_R(z) $ defines a distribution $\mathcal{N}_R$ of
dimension $\mu_R$ called nullity distribution of $R$.

Any  vector field belonging to the nullity distribution is called  a nullity vector field.
 \end{defn}

\begin{prop}\label{cartan.nul} The nullity distribution $\N_R$ has the following properties:
~\par
\begin{description}
   \item[(a)]$\mathcal{N}_R\neq \phi$.

  \item[(b)]$\mathcal{N}_R\subseteq \mathcal{N}_\mathfrak{R}$, where $\mathcal{N}_\mathfrak{R}$
  is the nullity distribution of the curvature $\mathfrak{R}$.

  \item[(c)]If $Z\in \mathcal{N}_{R}$, then $R(X,Y)Z=\mathcal{C}(F\mathfrak{R}(X,Y),Z)$.

  \item[(d)]If $S\in \mathcal{N}_{R}$, then $\mathfrak{R}=0$.

  \item[(e)] If $X\in \mathcal{N}_{R}$, then $[C,X]\in \N_R$ and consequently, $ [C,X]\in\mathcal{N}_{\mathfrak{R}}$.
\end{description}

\end{prop}
\begin{proof}~\par

\noindent \textbf{(b)} Let $X$ be  a  nullity  vector field.  We have 
\begin{eqnarray*}
               X\in \mathcal{N}_{R} &\Longrightarrow& R(X,Y)Z=0\quad \forall \, Y,Z \in \cppp \\
                &\Longrightarrow&  R(X,Y)S=0\quad \,\forall \, Y \in \cppp\\
                &\Longrightarrow& \mathfrak{R}(X,Y)=0\quad\,\,\,\,\, \forall \, Y \in \cppp\\
                 &\Longrightarrow&  X\in \mathcal{N}_{\mathfrak{R}}.
             \end{eqnarray*}
\noindent \textbf{(c)} Let $Z\in \mathcal{N}_{R} $,
then $Z\in \mathcal{N}_{\mathfrak{R}}$ and by Lemma \ref{bianchi} \textbf{(a)}, we have
 $$\mathfrak{S}_{X,Y,Z}\{R(X,Y)Z\}=\mathfrak{S}_{X,Y,Z}\{\mathcal{C}(F\mathfrak{R}(X,Y),Z)\}.$$
Since $R(Y,Z)X=R(Z,X)Y=0$ and $\mathfrak{R}(Y,Z)=\mathfrak{R}(Z,X)=0$,
 the result follows.\\

\noindent \textbf{(d)} Let $S\in  \mathcal{N}_{R} $, then  by \textbf{(c)},   we have
$R(X,Y)S=\mathcal{C}(F\mathfrak{R}(X,Y),S)$. Then, the result follows from (\ref{c(s)}) and  Lemma
\ref{car.curv.}. \\

\noindent \textbf{(e)}  Let  $X\in \N_R$.
 By the identity $D_CR=0$ \cite{r22}, we have
 $$(D_CR)(X,Y)=0,$$
 which leads to
 $$R(D_CX,Y)=0.$$
 Using (\ref{berwaldconn.}) and (\ref{cartanconn.}), we have $R([C,X],Y)=0$.
Since  $h$ is h(1), then $[C,h]=0$, from which  $[C,hX]=h[C,X]$. That is,
 $[C,hX]$ is horizontal.
 Hence, $[C,X]\in \N_R$. Consequently, by \textbf{(b)},
$[C,X]\in \mathcal{N}_{\mathfrak{R}}$.
\end{proof}

It is important to note that the converse of property \textbf{(b)} of Proposition \ref{cartan.nul} is not true in general, that is, $\N_\mathfrak{R}\not\subset \N_R$. This is shown by the next example in which the calculations are performed using MAPLE program.

\begin{example} \em{Let $M=\{x=(x_1,x_2,x_3,x_4)\in \mathbb{R}^4: x_4\neq0\}$, \\
$U=\{(x,y)\in\mathbb{R}^4\times \mathbb{R}^4:x_4\neq 0;\,y_i\neq 0,i=1,...,4 \}\subset TM$.\\
Let the energy function $E$ be defined on the open subset $U$ of $TM$ by:\\ $E=x_4y_1(y_2^3+y_3^3+y_4^3)^{1/3}$.
Then, we have:
\begin{eqnarray*}
  \Omega &=& \frac{1}{2(y_2^3+y_3^3+y_4^3 )^{2/3}}\{-( y_2^3+y_3^3+y_4^3) \,dx_1\wedge dx_4-y_1y_2^2 \,dx_2\wedge dx_4-y_1y_3^2 \,dx_3\wedge dx_4\\
   &&  -2x_4y_2^2 \,(dx_1\wedge dy_2 +dx_2\wedge dy_1)-2x_4y_3^2 \,(dx_1\wedge dy_3+dx_3\wedge dy_1)
   \end{eqnarray*}
   \begin{eqnarray*}
   &&-2x_4y_4^2 \,(dx_1\wedge dy_4+dx_4\wedge dy_1)\}-\frac{2x_4}{(y_2^3+y_3^3+y_4^3 )^{5/3}}\{y_1y_2(y_3^3+y_4^3) \,dx_2\wedge dy_2 \\
   && -y_1y_2^2y_3^2 \,(dx_2\wedge dy_3+dx_3\wedge dy_2)-y_1y_2^2y_4^2 \,(dx_2\wedge dy_4+dx_2\wedge dy_4)\\
   &&+y_1y_3(y_2^3+y_4^3) \,dx_3\wedge dy_3-y_1y_3^2y_4^2 \,(dx_3\wedge dy_4+dx_4\wedge dy_3)\\&&+y_1y_4(y_3^3+y_4^3) \,dx_4\wedge dy_4  \}.
\end{eqnarray*}
The identity $i_S\Omega=-dE$ gives the following  non-vanishing  coefficients of the canonical spray $S^i$:
$$S^2=\frac{3y_2y_4}{4x_4},\quad\quad S^3=\frac{3y_3y_4}{4x_4},\quad\quad S^4=-\frac{y_2^3+y_3^3-2y_4^3}{4x_4y_4}.$$
The non-vanishing coefficients of  Barthel connection $\Gamma^i_j$ are:
$$\Gamma^2_2=\frac{3y_4}{4x_4},\quad\quad \Gamma^2_4=\frac{3y_2}{4x_4},\quad\quad \Gamma^3_3=\frac{3y_4}{4x_4},\quad\quad \Gamma^3_4=\frac{3y_3}{4x_4},$$
$$\Gamma^4_2=-\frac{3y_2^2}{4x_4y_4},\quad\quad \Gamma^4_3=-\frac{3y_3^2}{4x_4y_4},\quad\quad \Gamma^4_4=\frac{y_2^3+y_3^2+4y_4^3}{4x_4y_4^2}.$$
The independent non-vanishing components of the curvature $\mathfrak{R}^i_{jk}$ of  Barthel connection are:
$$\mathfrak{R}^2_{23}=\frac{9y_3^2}{16x_4^2y_4},\quad\quad \mathfrak{R}^2_{24}=-\frac{3(y_2^3+y_3^2+5y_4^3)}{16x_4^2y_4^2},$$
$$\mathfrak{R}^3_{23}=-\frac{9y_2^2}{16x_4^2y_4},\quad\quad \mathfrak{R}^3_{34}=-\frac{3(y_2^3+y_3^2+5y_4^3)}{16x_4^2y_4^2},$$
$$\mathfrak{R}^4_{24}=\frac{3y_2^2(y_2^3+y_3^2+5y_4^3)}{16x_4^2y_4^4},\quad\quad \mathfrak{R}^4_{34}=\frac{3y_3^2(y_2^3+y_3^2+5y_4^3)}{16x_4^2y_4^4}.$$

Now, let $X\in \N_{\mathfrak{R}}$, then $X$ can be written in the form $X=X^1h_1+X^2h_2+X^3h_3+X^4h_4$, where $X^1, X^2, X^3, X^4$  are the  components of the nullity vector $X$ with respect to the basis   $\{h_1, h_2, h_3,h_4\}$  of the horizontal space, where  $h_i:=\frac{\partial}{\partial x^i}-\Gamma^m_i\frac{\partial}{\partial y^m}$, $i,m=1,...,4$.  The equation  $\mathfrak{R}(X,Y)=0$, $\forall\, Y\in H(TM)$, is written locally in the form
  $$X^j\mathfrak{R}^i_{jk}=0.$$
This is equivalent to the   system of equations:
{\vspace{-6pt}}$$3y_3^2X^3-(y_2^3+y_3^3+5y_4^3)X^4=0,$$
{\vspace{-15pt}}$$y_3^2X^2=0,$$
{\vspace{-10pt}}$$y_2^2X^3=0.$$
From the above system, we have $X^1=t_1, \,t_1\in \mathbb{R} $ and  $X^2=X^3=0$. Then, we get $(y_2^3+y_3^3+5y_4^3)X^4=0$. Now, we have two cases, either $y_2^3+y_3^3+5y_4^3=0$ or $y_2^3+y_3^3+5y_4^3\neq 0$. Firstly, if  $y_2^3+y_3^3+5y_4^3\neq 0$, then $X^4=0$ and thus $\mu_\mathfrak{R}=1$. Secondly, if $y_2^3+y_3^3+5y_4^3=0$, then     $X^4=t_4, \,t_4\in \mathbb{R}$ and thus $X=t_1h_1+t_4h_4$ and $\mu_\mathfrak{R}=2$. We will be interested in the second case.

Calculations using MAPLE give the coefficients of Cartan connection $\Gamma^i_{jk}$ and so the components of the  h-curvature tensor $R^h_{ijk}$.  Taking into account that $y_2^3+y_3^3+5y_4^3=0$, the independent non-vanishing components $R^h_{ijk}$ are as follows:
\begin{eqnarray*}
R^2_{123}&=&\frac{-9y_3^2}{32x_4^2y_1y_4},\quad\quad R^3_{123}=\frac{9y_2^2}{32x_4^2y_1y_4},\quad\quad R^1_{223}=\frac{-9y_1y_2y_3^2}{64x_4^2y_4^4}\\
 R^2_{223}&=&\frac{-9y_2^2y_3^2}{128x_4^2y_4^4},\quad\quad R^3_{223}=\frac{-9y_2(4y_2^3+6y_3^3)}{256x_4^2y_4^4},\quad\quad R^4_{223}=\frac{-45y_2y_3^2}{128x_4^2y_4^3},\\
R^3_{224}&=&\frac{-108y_2y_3}{256x_4^2y_4^2},\quad\quad\, R^4_{224}=\frac{-3y_2(y_3^3(2y_3^3-30y_4^3)-y_2^3(2y_2^3+14y_4^3)-20y_4^6)}{256x_4^2y_4^7},\\
R^3_{234}&=&\frac{27y_2^2}{64x_4^2y_4^2},\quad\quad R^4_{234}=\frac{27y_2^2y_3^2}{64x_4^2y_4^4},\quad\quad
R^1_{323}=\frac{9y_1y_2^2y_3}{64x_4^2y_4^4},\quad\quad
R^2_{323}=\frac{9y_3(y_2^3-8y_4^3)}{128x_4^2y_4^4},\\
 R^3_{323}&=&\frac{9y_2^2y_3^2}{128x_4^2y_4^4},\quad\quad
 R^4_{323}=\frac{45y_2^2y_3}{128x_4^2y_4^3},\quad\quad R^2_{324}=\frac{27y_3^2y_4^3}{64x_4^2y_4^5},\quad\quad R^4_{324}=\frac{-27y_2^2y_3^2}{64x_4^2y_4^4}, \\
R^2_{334}&=&\frac{-27y_2y_3y_4^3}{64x_4^2y_4^5},\quad\quad  R^4_{334}=\frac{3y_3(y_3^3(-3y_2^3+4y_4^3)+5y_4^3(4y_4^3+5y_2^3)-3y_2^6)}{256x_4^2y_4^7},\\ R^3_{423}&=&\frac{-9y_2^2}{32x_4^2y_4^2},\quad\quad  R^3_{424}=\frac{27y_2^2y_3}{64x_4^2y_4^3},\\
  R^2_{424}&=&\frac{-3(y_3^3(4y_2^3+38y_4^3)+2y_2^3(2y_2^3+11y_4^3)+10y_2^6)}{256x_4^2y_4^6},\quad\quad R^2_{432}=\frac{-9y_3^2}{32x_4^2y_4^2},\\
 R^2_{434}&=&\frac{27y_2y_3^2}{64x_4^2y_4^3},\quad\quad
  R^3_{434}=\frac{-34(y_3^6+22y_3^3y_4^3+y_2^3y_3^3+23y_2^3y_4^3-3y_2^6+10y_4^6)}{256x_4^2y_4^6}.
\end{eqnarray*}
Now, let $X\in \N_{{R}}$.  The equation  ${R}(X,Y)Z=0$, $\forall\, Y,Z\in H(TM)$, is written locally in the form
  $$X^j{R}^h_{ijk}=0.$$
This is equivalent to   the system of equations:
$$y_2(5y_2^3+7y_3^3+9y_4^3)X^3+12y_2y_3X^4=0,$$
$$y_3^2X^2=0,$$
$$y_2^2X^3=0.$$
The above system  has the solution  $X^1=t_1', t_1'\in \mathbb{R}$ and  $X^2=X^3=X^4=0$. Thus,  $X=t_1' h_1$ and  $\mu_{R}=1$. So, the dimension of $\N_R=1$ and the dimension of $\N_\mathfrak{R}=2$, consequently, $\N_\mathfrak{R}\not\subset \N_R$.    \qed
}
\end{example}

Nevertheless, we have some cases in which $\N_\mathfrak{R}\subset \N_R$ as the case of Landesberg spaces satisfying certain conditions.
\begin{defn}\cite{szilasi}
A Finsler space is called Landesberg if the second Cartan tensor vanishes: $\mathcal{C}'=0$ or,  equivalently, if $P=0$.
\end{defn}
\begin{thm}  Let $(M,E)$ be a Landesberg space. If, for all $X\in\N_\mathfrak{R},\,\, \overcirc{D}_{JZ}X\in \N_\mathfrak{R}$, then $\N_\mathfrak{R}\subset \N_R$ and hence $\N_\mathfrak{R}=\N_R$.
\end{thm}
\begin{proof}
Let $(M.E)$ be a Landesberg space.  Then,  using Lemma \ref{car.r,p,q}, we get
$$R(X,Y)Z=\overcirc{R}(X,Y)Z+\mathcal{C}(F\mathfrak{R}(X,Y),Z).$$
Let $X\in \N_\mathfrak{R}$, by the above equation and the fact that\, $\overcirc{R}(X,Y)Z=(\,\overcirc{D}_{JZ}\mathfrak{R})(X,Y)$ \cite{Nabil.1},  then, $R(X,Y)Z=-\mathfrak{R}(\,\overcirc{D}_{JZ}X,Y)$. Since $\, \overcirc{D}_{JZ}X\in \N_\mathfrak{R}, \forall X\in \N_\mathfrak{R}$, then $R(X,Y)Z=0$ and then  $X\in \N_R$. Consequently, $\N_\mathfrak{R}\subset \N_R$ and hence $\N_\mathfrak{R}=\N_R$.
\end{proof}

\begin{thm}Let $\mu_R$ be constant on an open subset $U$ of $TM$. Then,
the nullity distribution $z\mapsto \mathcal{N}_R(z)$ is completely integrable on $U$.
\end{thm}

\begin{proof} To prove this theorem we have to show that if  $X, Y \in \N_R$,
then $[X, Y] \in \N_R$. So, let $X, Y \in \N_R$ and $Z\in H(TM)$.
This implies that $X$ and $Y$ are horizontal and  $X, Y \in \N_\mathfrak{R}$. Then, by Lemma \ref{bianchi} \textbf{(e)}, we have
$$\mathfrak{S}_{X,Y,Z}\{(D_XR)(Y,Z)\}=\mathfrak{S}_{X,Y,Z}\{P(X,F\mathfrak{R}(Y,Z))\}.$$
Since $X, Y \in \N_\mathfrak{R}$, then $\mathfrak{R}(X,Y)=\mathfrak{R}(Y,Z)=\mathfrak{R}(Z,X)=0$.
Making use of  Lemma \ref{car.[]} and  the fact that $R$ is semi-basic and $T(hX,hY)=\mathfrak{R}(X,Y)$, we have
\begin{eqnarray*}
   0&=&\mathfrak{S}_{X,Y,Z}\{(D_XR)(Y,Z)\}\\
   &=&\mathfrak{S}_{X,Y,Z}\{D_XR(Y,Z)-R(D_XY,Z)-R(Y,D_XZ)\}\\
   &=&-R(D_XY,Z) -R(Z,D_YX) \\
   &=& R(D_XY-D_YX,Z)\\
   &=& R([X,Y]+\mathfrak{R}(X,Y),Z)\\
   &=& R([X,Y],Z)+R(\mathfrak{R}(X.Y),Z)\\
   &=& R([X,Y],Z),\,\,\,\, \forall Z\in H(TM).
\end{eqnarray*}
It remains to show that $[X,Y]$ is horizontal. In fact, as $\mathfrak{R}(X,Y)=-v[hX,hY]$ \cite{Nabil.2},  $0=\mathfrak{R}(X,Y)=-v[X,Y]$, and hence $[X,Y]$ is horizontal. Hence, we have $[X,Y]\in \N_R$.
\end{proof}

\begin{rem}\em{It should be noted that the nullity distribution $\N_\mathfrak{R}$ of the curvature of  Barthel  connection is completely integrable as has been proved in \cite{Nabil.2}.}
\end{rem}

We have seen that if the index of nullity $\mu_R$ is constant, then  the nullity  distribution $\N_R$ is completely integrable. Then, according to  the Frobenius theorem, there exists a foliation of $TM$ by $\mu_R(z)$-dimensional maximal connected submanifolds which are called  the leaves, such that $\N_R(z)$ is the tangent space to the leaf at $z\in TM$. In this case we call the foliation induced by the nullity distribution  the nullity foliation.

\begin{thm} The leaves of the nullity foliations  of $\mathcal{N}_\mathfrak{R}$ and $\mathcal{N}_R$
are auto-parallel submanifolds.
\end{thm}
\begin{proof}
To prove that  $\N_R$ is auto-parallel  with respect to Cartan connection, we have to show that if $X,Y \in \N_{R}$,
 then  $D_XY\in \N_{R}$.

 Let $X,Y \in \N_{R}$, then $X,Y \in \N_{\mathfrak{R}}$ and $X,Y\in H(TM)$.
 As $Dh=0$, then $D_{X}hY=hD_{X}Y$,
 i.e., $D_{X}Y\in H(TM)$. By Lemma \ref{bianchi} \textbf{(e)}, we have
$$\mathfrak{S}_{X,Y,Z}\{(D_X{R})(Y,Z)\}=0.$$
Consequently
 $$\mathfrak{S}_{X,Y,Z}\{{R}(D_XY,Z)\}=0.$$

\noindent Hence  ${R}(D_XY,Z)=0 \,\, \forall Z\in \cppp$ and  $D_XY\in \N_{R}$.\\

 Similarly, we  show that if $X,Y\in \N_\mathfrak{R}$, then $D_XY\in \N_\mathfrak{R}$.
  By Lemma \ref{bianchi} \textbf{(d)}, we have
$$\mathfrak{S}_{X,Y,Z}\{(D_X\mathfrak{R})(Y,Z)\}=\mathfrak{S}_{X,Y,Z}\mathcal{C}\{(F\mathfrak{R}(X,Y),Z)\}.$$
Since $X,Y\in\N_\mathfrak{R}$, then
$$\mathfrak{S}_{X,Y,Z}\{(D_X\mathfrak{R})(Y,Z)\}=0.$$
Consequently,  $\mathfrak{R}(D_XY,Z)=0 \,\, \forall Z\in \cppp$ and $D_XY\in \N_\mathfrak{R}$.
\end{proof}
It is well known that the concepts of auto-parallel submanifold and totally geodesic
submanifold coincide in  Riemannian geometry  \cite{kobayashi}.
 This is not true,  in general. However,
  every auto-parallel submanifold   is totally geodesic \cite{E.cartan}. So,
 we have the following corollary.
\begin{cor}
The leaves of the  nullity foliations $\N_\mathfrak{R}$ and $\N_R$ are totally geodesic submanifolds.
\end{cor}

\begin{defn}\label{hisotropic}{{\cite{r14}, \cite{special FS} }} A Finsler space $(M,E)$, where  $ \dim M\geq 3$,
is said to be $h$-isotropic  if there exists a scalar function  $k_{o}$ such that the $h$-curvature tensor
$R$ of Cartan connection  has the form
$$R({X},{Y}){Z}=k_{o} \{g({X},{Z})
{Y}-g({Y},{Z}){X} \},\,\,\, \forall \, X,Y,Z\in \cppp.$$
\end{defn}

\begin{thm} For  an  $h$-isotropic Finsler space,   the index of nullity $\mu_R$ takes its maximal  value, i.e. $\mu_R=n$.
\end{thm}

\begin{proof}
Let $X$ be a non zero nullity vector in  $\N_R$ and $Y,Z,W\in \cppp$. Then, by  Definition  \ref{hisotropic}, we have
\begin{eqnarray*}
  0 &=& k_{o} \{g({X},{Z})
{Y}-g({Y},{Z}){X} \} \\
   &=&k_{o} \{g(g({X},{Z}){Y},W)-g(g({Y},{Z}){X},W) \}  \\
  &=& k_{o} \{g({Y},W)g({X},{Z})-g({X},W)g({Y},{Z}) \}.
\end{eqnarray*}
 As  $g$ is a metric on $TM$, its  trace  is thus $2n$.  Taking the trace with respect to the pair $Y$ and $W$, we get
$$k_{o} \{2ng({X},{Z})-g({X},Z) \}=0,$$
Again, taking the trace of the above equation, we have
$$ 2n(2n-1)k_{o}=0.$$
which gives   $k_{o}=0$. Consequently, $R=0$ and  hence $\mu_R=n$.
\end{proof}

\begin{defn}{{\cite{szilasi}, \cite{special FS} }} A Finsler space $(M,E)$,
is said to be Berwald  space  if  the $hv$-curvature tensor\,
$\overcirc{P}$ of Berwald connection  vanishes or, equivalently, $D_{hX}\mathcal{C}=0$ for all $X\in \cppp$.
\end{defn}

\begin{thm} For  a Berwald space,   the index of nullity $\mu_\mathfrak{R}$ of  $\N_\mathfrak{R}$ takes its maximal value if and only if   the index of nullity $\mu_{R}$ of  $\N_{R}$ takes its maximal value.
\end{thm}

\begin{proof} Let  $(M,E)$ be a Berwald space and so $\mathcal{C}'=0$ \cite{szilasi}.  Hence, by
 Lemma \ref{car.r,p,q} \textbf{(a)}, the h-curvature of Cartan connection is written in the form
\begin{equation}\label{star}
R(X,Y)Z  = \overcirc{R}(X,Y)Z+\mathcal{C}(F\mathfrak{R}(X,Y),Z).
\end{equation}
Now, let $\mu_\mathfrak{R}=n$.  Then $\mathfrak{R}=0$, which is  equivalent to \, $\overcirc{R}=0$ \cite{Nabil.1}. Thus, Equation (\ref{star}) yields  $R=0$. Consequently, $\mu_{R}=n$.

Conversely,  let $\mu_{R}=n$. Hence,  by (\ref{star}), \,$\overcirc{R}(X,Y)Z+\mathcal{C}(Z,F\mathfrak{R}(X,Y))=0$. Setting $Z=S$ in this equation, we have  $\overcirc{R}(X,Y)S=0$.  But \,$\overcirc{R}(X,Y)S=\mathfrak{R}(X,Y)$ \cite{Nabil.1}. Thus, $\mathfrak{R}=0$, consequently, $\mu_\mathfrak{R}=n$.
\end{proof}

\Section{Nullity distribution of Cartan hv-curvature}

 In this section, we  study the nullity distribution of the hv-curvature of  Cartan connection. We show that the nullity distribution $\N_P$ of the hv-curvature $P$ is not  completely integrable. We impose a certain condition to make $\N_P$ completely integrable. We present a class of Finsler spaces which guarantees the  possibility of such a condition.

\begin{defn} Let $P$ be the hv-curvature of Cartan connection.
The  nullity space  of $P$ at a point $z\in TM$ is the subspace of $H_z(TM)$ defined by
$$\mathcal{N}_P(z):=\{X\in H_z(TM) : \,  P(X,Y)=0, \, \,\forall Y\in T_zTM\}.$$
The dimension of $\mathcal{N}_P(z)$, denoted by $\mu_P(z)$, is the index of nullity of $P$ at $z$.
\end{defn}
\begin{prop}\label{p.nullity}The  nullity distribution  of $P$ has the following properties:
\begin{description}
\item[(a)]$ \N_P\neq\phi$.

  \item[(b)]$S\in \N_P$.

  \item[(c)] If $X\in \N_P(z)$, then $\mathcal{C}'(X,Y)=0,\, \forall\, Y\in T_zTM.$

  \item[(d)] If $X, Y \in \N_P\cap \N_\mathfrak{R}$,   then
$R(X,Y)Z=\mathcal{C}'([X,Y],Z).$
\end{description}
\end{prop}
\begin{proof}~\par

\noindent \textbf{(b)} Follows from the fact that $P(S,X)Y=P(X,S)Y=0$  (Lemma \ref{car.curv.}).\\

\noindent \textbf{(c)} Let $X\in \N_P(z)$,
\begin{eqnarray*}
  X\in \N_P(z) &\Longrightarrow&P(X,Y)Z=0\quad \forall\, Y,Z\in T_zTM  \\
   &\Longrightarrow& P(X,Y)S=0\quad \,\forall \,Y\in T_zTM \\
   &\Longrightarrow& \mathcal{C}'(X,Y)=0 \quad \,\ \,\,\forall\, Y\in T_zTM.
\end{eqnarray*}
\noindent \textbf{(d)}
Let $X$, $Y \in \N_P\cap \N_\mathfrak{R}$. Then, by  Proposition   \ref{p.nullity} \textbf{(c)}, Lemma \ref{car.r,p,q} \textbf{(a)} and the identity \, $\overcirc{R}(X,Y)Z=(\,\overcirc{D}_{JZ}\mathfrak{R})(X,Y)$ \cite{Nabil.1}, we have
$$ R(X,Y)Z  = (D_{hX}\mathcal{C}')(Y,Z)-(D_{hY}\mathcal{C}')(X,Z). $$
By Lemma \ref{car.[]} \textbf{(c)} and the fact that $\mathcal{C}'$ is semi-basic, we get
$$R(X,Y)Z=\mathcal{C}'([hX,hY],Z).$$
Hence, the result follows.
\end{proof}

\begin{thm}
 For a Landesberg space, the nullity distributions $\N_R$ and  $\N_{R^\circ}$  coincide, where $\N_{{R^\circ}}$ is the nullity distribution of the h-curvature\, $\overcirc{R}$ of Berwald connection.
\end{thm}

\begin{proof}
Let $(M,E)$ be a Landesberg space. Then, the hv-curvature $P$ of Cartan connection  vanishes and thus $\N_P=H(TM)$. Consequently, $\mathcal{C}'=0$, by Proposition \ref{p.nullity}~\textbf{(c)}. Hence, by  Lemma \ref{car.r,p,q} \textbf{(a)}, we get
$$R(X,Y)Z=\overcirc{R}(X,Y)Z+\mathcal{C}(F\mathfrak{R}(X,Y),Z).$$
Let $X\in \N_R$, then $X\in \N_\mathfrak{R}$ and thus $\mathfrak{R}(X,Y)=0$, hence, $X\in \N_{{R^\circ}}$. Consequently, $\N_R\subseteq \N_{{R^\circ}}$. Conversely, let $X\in \N_{{R^\circ}}$, then $X\in \N_\mathfrak{R}$ \cite{Nabil.1} and thus $\mathfrak{R}(X,Y)=0$, hence, $X\in \N_{{R}}$. Consequently, $\N_{{R^\circ}}\subseteq \N_{{R}}$.
\end{proof}

The nullity distribution $\N_P$ is not in general completely integrable as shown by the following example.

\begin{example} \em{Let $M=\{x=(x_1,x_2,x_3)\in \mathbb{R}^3:x_2\neq0\}$, \\
$U=\{(x,y)\in\mathbb{R}^3\times \mathbb{R}^3: x_2\neq 0;\,y_1,y_2\neq 0\}\subset TM$.\\
Let the energy function $E$ be defined on U by: $E=e^{-x_1}(e^{-x_1x_3}y_1^2y_3+x_2y_2^3)^{2/3}$.
\vspace{5pt}

For simplicity, let $\sigma_1:=e^{-x_1x_3}y_1^2y_3+x_2y_2^3$,  $\sigma_2:=7e^{-x_1x_3}y_1^2y_3+12x_2y_2^3$ and $\sigma_3:=e^{x_1x_3}(5e^{-x_1x_3}y_1^2y_3+3x_2y_2^3)$.
Then, the non-vanishing components $P^h_{ijk}$ of the hv-curvature tensor $P$ are:
$$P^1_{111}=\frac{-3x_2y_2^3}{32y_1\sigma_1},\quad\quad P^2_{111}=\frac{-y_2\sigma_2}{32y_1^2\sigma_1},
\quad\quad P^3_{111}=\frac{-9x_2y_2^3y_3}{32y_1^2\sigma_1},\quad\quad P^1_{112}=\frac{3x_2y_2^2}{32\sigma_1}=P^1_{121},$$
$$P^2_{112}=\frac{\sigma_2}{32y_1\sigma_1}=P^2_{121},\quad\quad P^3_{112}=\frac{9x_2y_2^2y_3}{32y_1\sigma_1}=P^3_{121},\quad\quad
P^1_{122}=\frac{-3x_2y_1y_2}{32\sigma_1},$$
$$P^2_{122}=\frac{-\sigma_2}{32y_2\sigma_1},\quad\quad P^3_{122}=\frac{-x_2y_2y_3}{32\sigma_1},\quad\quad P^1_{211}=\frac{3x_2y_2^2}{32\sigma_1},\quad\quad P^2_{211}=\frac{x_2y_2^3}{16y_1\sigma_1},$$
$$P^3_{211}=\frac{3x_2y_2^2\sigma_3}{16y_1^3\sigma_1},\quad\quad P^1_{221}=\frac{-3x_2y_1y_2}{32\sigma_1}=P^1_{212},\quad\quad P^2_{221}=\frac{-3x_2y_2^2}{16\sigma_1}=P^2_{212}
,$$
$$P^3_{221}=\frac{-3x_2y_2}{16y_1^2\sigma_1}=P^3_{212},\quad\quad P^1_{222}=\frac{3x_2y_1^2}{32\sigma_1},\quad\quad P^2_{222}=\frac{3x_2y_1y_2}{16\sigma_1},\quad\quad
P^3_{222}=\frac{3x_2\sigma_3}{16\sigma_1},$$
$$P^2_{311}=\frac{y_1y_2e^{-x_1x_3}}{32\sigma_1},\quad\quad P^3_{311}=\frac{-3x_2y_2^3}{32y_1\sigma_1},\quad\quad P^2_{312}=\frac{-x_2y_1^2e^{-x_1x_3}}{32y_1^3\sigma_1}=P^2_{321},$$
$$P^3_{312}=\frac{3x_2y_2^2}{32\sigma_1}=P^3_{321},\quad\quad P^2_{322}=\frac{y_1^3e^{-x_1x_3}}{32y_2\sigma_1},\quad\quad P^3_{322}=\frac{-3x_2y_1y_2}{32\sigma_1}.$$
\vspace{5pt}
Now, let $X\in \N_{P}$.  The equation  ${P}(X,Y)Z=0, \forall\,\, Y,Z\in H(TM)$,  is written locally in the form
  $$X^j{P}^h_{ijk}=0.$$
This yields  the  system of equations
$$y_2X^1-y_1X^2=0.$$
 Thus, the solution of the above system is  $X^1=t_1, X^2=\frac{y_2}{y_1}t_1$ and $X^3=t_3,\,t_1,t_3\in \mathbb{R}$.    Hence,  $X=t_1( h_1+\frac{y_2}{y_1}h_2)+t_3h_3$ and  $\mu_{P}=2$. Now, let $X, Y\in \N_P$ be such that  $X=h_1+\frac{y_2}{y_1}h_2$ and $Y=h_3$. By  simple calculations,  the bracket $[X,Y]=[h_1+\frac{y_2}{y_1}h_2,h_3]=-\frac{1}{2}y_1\frac{\partial}{\partial y_1}+y_3\frac{\partial}{\partial y_3}$,  which is vertical. Consequently, the nullity distribution $\N_P$ is not complectly integrable. \qed
}
\end{example}

Nevertheless, we have

\begin{thm}\label{p.is.c.i.}
Let  $\mu_P$ be constant on an open subset $U$ of $TM$.  The  nullity distribution
 $\N_P$ is completely integrable on $U$ if and only if,  for all $X,Y \in \N_P$,
 $ \mathfrak{R}(X,Y)=0$ and $ (D_{JZ}R)(X,Y)=R(Y,F{\mathcal{C}}(X,Z))-R(X,F{\mathcal{C}}(Y,Z))$.
\end{thm}
\begin{proof} Let  $X, Y\in \N_P$. Then,   $\mathfrak{R}(X,Y)=0$  and $ (D_{JZ}R)(X,Y)=R(Y,F{\mathcal{C}}(X,Z))-R(X,F{\mathcal{C}}(Y,Z))$, $ \forall Z\in \cppp$. As $\mathfrak{R}(X,Y)=0$, then the bracket
 $[hX,hY]$ is horizontal. Making use of Lemma \ref{bianchi} \textbf{(f)} and Lemma \ref{car.[]} \textbf{(c)}, we get
 \begin{eqnarray*}
  ( D_{hX}P)(Y,Z)-(D_{hY}P)(X,Z)=0&\Longrightarrow&P(D_{X}Y-D_{Y}X,Z)=0 \\
    &\Longrightarrow& P([X,Y]+\mathfrak{R}(X,Y),Z)=0 \\
    &\Longrightarrow& P([X,Y],Z)=0\\
    &\Longrightarrow&[X,Y]\in \N_P.
 \end{eqnarray*}
Hence $\N_P$ be completely integrable.

Conversely, let $\N_P$ be completely integrable. Then, if $X,Y\in \N_P$,  the bracket $[hX,hY]$ is horizontal, thus, $\mathfrak{R}(X,Y)=0$. Also,  by Lemma \ref{bianchi} \textbf{(f)} and the fact   $P([hX,hY],Z)=( D_{hX}P)(Y,Z)-(D_{hY}P)(X,Z)=0$,  we have $ (D_{JZ}R)(X,Y)=R(Y,F{\mathcal{C}}(X,Z))-R(X,F{\mathcal{C}}(Y,Z)),\,  \forall X,Y \in \N_P, \, \forall Z\in \cppp$.
\end{proof}

\begin{rem}
\em{ The  class of Finsler spaces with vanishing  h-curvature  satisfy the conditions of Theorem \ref{p.is.c.i.}. Consequently, for such spaces,  $\N_P$ is completely integrable. }
\end{rem}

Moreover, we have

\begin{prop} A sufficient  condition for  $\N_P$ to be  completely integrable is that  $$\N_P\subset N_R.$$
\end{prop}
\begin{proof}
Let $\N_P\subset N_R$ and $X,Y \in \N_P, \, Z\in \cppp$. Then,  $X,Y\in \N_R$ and hence $X,Y\in \N_\mathfrak{R}$, consequently, $\mathfrak{R}(X,Y)=0$.  Also, by Lemma \ref{bianchi} \textbf{(f)}, we have $ (D_{JZ}R)(X,Y)=R(Y,F{\mathcal{C}}(X,Z))-R(X,F{\mathcal{C}}(Y,Z)),\,  \forall X,Y \in \N_P, \, \forall Z\in \cppp$. Hence, by Theorem \ref{p.is.c.i.}, $\N_P$ is completely integrable.
\end{proof}


\Section{Nullity distribution of Cartan v-curvature}

In this section, we  study  the nullity distribution of the v-curvature  $Q$ of  Cartan connection.
The  nullity distribution  of $Q$ is defined in a similar manner as that  of $R$ (Definition \ref{nr} )

\begin{prop}\label{Q.nullity}The  nullity distribution of $Q$ satisfies:
\begin{description}
\item[(a)]$\N_Q\neq \phi$.
  \item[(b)]$S\in \N_Q$.
  \item[(c)] If $Z\in \N_Q$, then $Q(X,Y)Z=0, \forall X,Y\in \cppp$. \\That is,  $Q(X,Y)Z$ vanishes whenever $X$, $Y$ or $Z$ is a $Q$-nullity vector field.
  \item[(d)] If $X,Y\in\N_Q(z)$, then $F[JX,JY]\in\N_Q$.
\end{description}
\end{prop}
\begin{proof}~\par

\noindent \textbf{(b)} Follows from the fact that  $Q(S,X)Y=0$. (Lemma \ref{car.curv.} \textbf{(d)})

\noindent \textbf{(c)} Follows from  Lemma \ref{bianchi} \textbf{(b)}.

\noindent \textbf{(d)} Let $ X,Y\in \N_Q$, then Propositions  \ref{Q.nullity} and Lemma \ref{bianchi} \textbf{(h)} lead to
\begin{eqnarray*}
  0 &=& (D_{JX}Q)(Y,Z)+(D_{JY}Q)(Z,X)+(D_{JZ}Q)(X,Y)\\
   &=&  -Q(D_{JX}Y,Z)-Q(Z,D_{JY}X)\\
   &=&Q(D_{JY}X-D_{JX}Y,Z)\\
   &=&Q(F[JX,JY],Z).
\end{eqnarray*}
Since $[JX,JY]$ is vertical and $FJ=h$, hence $F[JX,JY]$ is horizontal. Consequently,  $F[JX,JY]\in\N_Q$.
\end{proof}

The nullity distribution $\N_Q$ is not in general completely integrable as shown  by the following example.

Let
$E=x_4y_1(y_2^3+y_3^3+y_4^3)^{1/3}$.
Then, we have:

\begin{example} \em{$M=\{x=(x_1,x_2,x_3,x_4)\in \mathbb{R}^4: x_2\neq0\}$, \\
$U=\{(x,y)\in\mathbb{R}^4\times \mathbb{R}^4:x_2\neq 0;\,y_1,y_3,y_4\neq 0\}\subset TM$.\\
Let the energy function $E$ be defined on $U$ by $E=x_2y_1^2e^{-y_3/y_4}+y_2^2$.

\vspace{5pt}
Then, the independent non-vanishing components of the v-curvature $Q^h_{ijk}$ of Cartan connection  are:
   $$Q^3_{113}=\frac{-y_3}{2y_1^2y_4 },\quad\quad Q^4_{113}=\frac{-1}{2y_1^2},\quad\quad Q^3_{114}=\frac{y_3^2}{2y_1^2y_4^2 },\quad\quad Q^4_{114}=\frac{y_3}{2y_1^2y_4},$$
  $$Q^3_{134}=\frac{-y_3}{2y_1y_4^2 },\quad\quad Q^4_{134}=\frac{-1}{2y_1y_4},\quad\quad Q^3_{313}=\frac{-1}{2y_1y_4},\quad\quad Q^1_{314}=\frac{y_3}{4y_4^3},$$
   $$Q^1_{331}=\frac{-1}{4y_4^2 },\quad\quad Q^1_{334}=\frac{-y_1}{4y_4^3 },\quad\quad Q^3_{334}=\frac{-1}{2y_1^2},\quad\quad Q^3_{341}=\frac{-y_3}{2y_1^2y_4},$$
  $$Q^1_{413}=\frac{y_3}{4y_4^3},\quad\quad,Q^3_{413}=\frac{y_3}{y_1y_4^2 },\quad\quad Q^1_{414}=\frac{-y_3^2}{4y_4^4 },\quad\quad Q^3_{414}=\frac{-y_3}{y_1y_4^2 },$$
   $$Q^4_{414}=\frac{-y_3}{2y_1y_4^2 },\quad Q^4_{413}=\frac{1}{2y_1^2y_4 },\quad Q^4_{423}=\frac{1}{2y_1y_4},\quad Q^1_{434}=\frac{y_1y_3}{4y_4^4 },\quad Q^3_{434}=\frac{y_3}{y_4^3 }.$$
  Now, let $X\in \N_{Q}$, the equation  ${Q}(X,Y)Z=0$, $\forall \, Y,Z\in H(TM)$,  is written locally in the form
  $$X^j{Q}^h_{ijk}=0.$$
This is equivalent to the   system of equations:
$$y_4X^3-y_3X^4=0,$$
$$y_4X^1-y_1X^4=0,$$
$$y_3X^1-y_1X^3=0.$$
From  the above system,   we have   $X^2=t,\, X^4=t',\, X^1=\frac{y_1}{y_4}t',\, X^3=\frac{y_3}{y_4}t',\,t,t'\in \mathbb{R}$.    Hence,  $X=th_2+t'(\frac{y_1}{y_4} h_1+\frac{y_3}{y_4}h_3+h_4)$ and  $\mu_{Q}=2$. Now, let $X, Y\in \N_Q$ be such that  $X=h_2$ and $Y=\frac{y_1}{y_4} h_1+\frac{y_3}{y_4}h_3+h_4$. Then,  the bracket $[X,Y]=[h_2,\frac{y_1}{y_4} h_1+\frac{y_3}{y_4}h_3+h_4]=-\frac{y_1y_2}{2x_2^2y_4}\frac{\partial}{\partial y_1}+\frac{y_1^2(5y_3-2y_4)}{4x_2y_4^2}e^{-y_3/y_4}\frac{\partial}{\partial y_2}+\frac{y_4}{2x_2^2}\frac{\partial}{\partial y_4}$,  which is vertical. Consequently, the nullity distribution $\N_Q$ is not complectly integrable.  \qed
}
\end{example}

Nevertheless, we have

\begin{thm}\label{q.is.c.i}
Let  $\mu_Q$ be constant on an open subset $U$ of \,$TM$. The  nullity distribution
 $\N_Q$  is completely integrable on $U$ if and only if, for all $X,Y \in\N_Q$, $ \mathfrak{R}(X,Y)=0$ and the tensor
    $$ A(X,Y,Z):=P(F{\mathcal{C}}(Z,X),Y)-(D_{JX}P)(Y,Z)
 -(D_{JZ}P)(X,Y), \, \forall Z\in \cppp $$
is symmetric  in $X$ and $Y$.
\end{thm}

\begin{proof}
Let $ X,Y\in \N_Q$. Then,  $ \mathfrak{R}(X,Y)=0$ and the tensor $A(X,Y,Z)$ is symmetric in the first two arguments.
   By  Lemma \ref{bianchi} \textbf{(g)}, we have
\begin{eqnarray}
\nonumber(D_{hX}Q)(Y,Z)&=&(D_{JY}P)(X,Z)-(D_{JZ}P)(X,Y)+P(F{\mathcal{C}}(X,Y),Z)\\
 \label{Q.c.i.1} &&-P(F{\mathcal{C}}(Z,X),Y)-Q(F\mathcal{C}'(X,Y),Z).
\end{eqnarray}
Interchange $X$ with $Y$ in the above equation, we get
\begin{eqnarray}
\nonumber (D_{hY}Q)(X,Z)&=&(D_{JX}P)(Y,Z)-(D_{JZ}P)(Y,X)+P(F{\mathcal{C}}(Y,X),Z)\\
\label{Q.c.i.2}  &&-P(F{\mathcal{C}}(Z,Y),X)-Q(F\mathcal{C}'(Y,X),Z).
\end{eqnarray}
Making use of the symmetry of $\mathcal{C}$ and $\mathcal{C}'$, Equations  (\ref{Q.c.i.1}) and (\ref{Q.c.i.2}) give
\begin{equation}\label{axyz}
(D_{hX}Q)(Y,Z)-(D_{hY}Q)(X,Z)=A(X,Y,Z)-A(Y,X,Z).
\end{equation}
Then, by the symmetry of $A(X,Y,Z)$ in $X$, $Y$, we get
$$Q(D_{hY}X-D_{hX}Y,Z)=0.$$
Consequently, it follows from Lemma \ref{car.[]} that
$$Q([X,Y]+\mathfrak{R}(X,Y),Z)=0.$$
Since $\mathfrak{R}(X,Y)=0$,  $[hX,hY]=[X,Y]$ is horizontal and so $[X,Y]\in \N_Q$.
  Consequently, $\N_Q(z)$ is completely integrable.

  Conversely, let $\N_Q$ be completely integrable. Then, for all $ X,Y \in \N_Q$,     the bracket $[hX,hY]\in \N_Q$, i.e., $[hX,hY]$ is horizontal and hence $\mathfrak{R}(X,Y)=0$. Moreover, by (\ref{axyz}) and the fact that $Q([hX,hY],Z)=0$,  the tensor $A(X,Y,Z)$ is symmetric in $X$ and $Y$.
\end{proof}
\begin{rem}
\em{ The  class of  Minkoweski spaces satisfies  the conditions of the above theorem. Consequently, a Minkoweski space has a  completely integrable  $\N_Q$. }
\end{rem}

\begin{defn}
Let $(M,E)$ be a Finsler manifold. The angular metric $\hbar$ on $TM$  is defined by
$$\hbar(X,Y)=g(X,Y)-\ell(X)\ell(Y),$$
where $g$ is the  metric tensor on $TM$ given by (\ref{metricg}) and $\ell(X):=\frac{1}{\sqrt{2E}}g(X,C)$.
\end{defn}
It should be noted that the trace of $\hbar$  is $(2n-1)$.
\begin{defn} A Finsler space $(M,E)$ of $\dim \geq 4$ is said to be $S_3$-like if
$$Q(X,Y,Z,W)=r\{\hbar(JX,JZ)\hbar(JY,JW)-\hbar(JX,JW)\hbar(JY,JZ)\},$$
where $Q(X,Y,Z,W)=g(Q(X,Y)Z,JW)$ and  $r$ is a scalar function.
\end{defn}

\begin{thm}
Let $(M,E)$ be an $S_3$-like space. Then, the index of nullity $\mu_Q$ takes its maximal value.
\end{thm}
\begin{proof}
Let $(M,E)$ be an $S_3$-like space and $X\in \N_Q$, then we have
$$r\{\hbar(JX,JZ)\hbar(JY,JW)-\hbar(JX,JW)\hbar(JY,JZ)\}=0.$$
Taking the trace with respect to $JX$ and $JZ$ , we get
$$(2n-2)r\hbar(JY,JW)=0.$$
Again, taking the trace of the above equation, we have
$$(2n-1)(n-1)r=0.$$
As $n\geq 4$, then $r=0$ and consequently  $Q=0$.
\end{proof}

\providecommand{\bysame}{\leavevmode\hbox
to3em{\hrulefill}\thinspace}
\providecommand{\MR}{\relax\ifhmode\unskip\space\fi MR }
\providecommand{\MRhref}[2]{%
  \href{http://www.ams.org/mathscinet-getitem?mr=#1}{#2}
} \providecommand{\href}[2]{#2}


\end{document}